\documentclass[11 pt, reqno]{amsart}

\usepackage{amssymb}
\usepackage{color}
\usepackage[allcolors = blue, colorlinks = true]{hyperref}
\usepackage{mathptmx}
\usepackage{comment}
\numberwithin{equation}{section}

\newtheorem{theorem}{Theorem}[section]
\newtheorem{corollary}[theorem]{Corollary}
\newtheorem{lemma}[theorem]{Lemma}

\theoremstyle{definition} 
\newtheorem{definition}[theorem]{Definition}
\theoremstyle{remark}  
\newtheorem{remark}[theorem]{Remark}

\newcommand{\R}{\mathbb{R}}
\newcommand{\Rn}{\R^n}
\newcommand{\il}{\Delta_{\infty}}
\newcommand{\ilN}{\Delta_{\infty}^N}
\newcommand{\ue}{u^{\epsilon}}
\newcommand{\Ve}{V^{\epsilon}}
\newcommand{\la}{\langle}
\newcommand{\ra}{\rangle}
\newcommand{\Om}{\Omega}
\newcommand{\loc}{\textnormal{loc}}

\DeclareMathOperator{\spt}{spt\,}

\DeclareMathOperator{\diverg}{div\,}

\newcommand{\abs}[1]{\left|#1\right|}

\newcommand{\norm}[1]{\left|\left|#1\right|\right|}

\newcommand{\divt}{\operatorname{div}}
\newcommand{\vp}{\varphi}
\newcommand{\ud}{\, d}
\newcommand{\half}{{\frac{1}{2}}}
\newcommand{\kom}[1]{}
\renewcommand{\kom}[1]{{\bf [#1]}}
\newcommand{\eps}{{\varepsilon}}

\begin{document}
\title[]{On the second order regularity of solutions to the parabolic $p$-Laplace equation}

\author{Yawen Feng}
\address{Department of Mathematics and Statistics, University of
Jyv\"askyl\"a, PO~Box~35, FI-40014 Jyv\"askyl\"a, Finland}
\email{yawen.y.feng@jyu.fi}
\email{mikko.j.parviainen@jyu.fi}

\author{Mikko Parviainen}
\address{}
\email{}

\author{Saara Sarsa}
\address{Department of Mathematics and Statistics, University of Helsinki, PO~Box~68, (Pietari Kalmin katu 5), FI-00014 University of Helsinki, Finland}
\email{saara.sarsa@helsinki.fi}

\begin{abstract}
In this paper, we study the second order Sobolev regularity of solutions to the parabolic $p$-Laplace equation. For any $p$-parabolic function $u$, we show that $D(\abs{Du}^{\frac{p-2+s}{2}}Du)$ exists as a function and belongs to  $L^{2}_{\text{loc}}$ with $s>-1$ and $1<p<\infty$. The range of $s$ is sharp.
\end{abstract}

\maketitle
\section{Introduction}
The elliptic $p$-Laplace equation has an extensive literature on the second order regularity. In contrast, the second order regularity for the  parabolic $p$-Laplace equation
\begin{equation}\label{eq:p-parab1}
    u_t=  {\rm div}(|Du|^{p-2}Du)
\end{equation}
 is much less studied. Throughout the paper we have $1<p<\infty$. In the elliptic case, one of the known estimates shows $W_{\text{loc}}^{1,2}$ regularity for the nonlinear expression of the gradient
 $$ |Du|^{\frac{p-2+s}2}Du $$
 proven by Dong, Peng, Zhang and Zhou \cite{dongpzz20} with $s>2-\min\{p+\frac n{n-1}, 3+\frac{p-1}{n-1}\}$, and then extended to $s>-1-\frac{p-1}{n-1}$ by the third author \cite{sarsa20}. The aim of this paper is to prove such a result to the parabolic $p$-Laplace equation. In other words, we prove in Theorem \ref{thm:MainTheorem} that for any weak or viscosity solution $u$ to \eqref{eq:p-parab1}, $D(\abs{Du}^{\frac{p-2+s}{2}}Du)$
exists, belongs to  $L^{2}_{\text{loc}}$ whenever $s>-1$, and this range is sharp.

In the parabolic case, Dong, Peng, Zhang and Zhou \cite{dongpzz20} proved for $p\in(1,3)$ that the weak or viscosity solution $u$ to \eqref{eq:p-parab1} locally belongs to $W^{2,2}$. This result is obtained as a special case from ours by selecting $s=2-p$.  Our result also contains as a special case Lindqvist's \cite{lindqvist08} result for  $ |Du|^{p-2}Du$ and $ |Du|^{\frac{p-2}{2}}Du$ in the range $p\geq2$. As a consequence, he also observed that the time derivative $u_t$ exists and belongs to a Sobolev space.  See also a recent paper by Cianchi and Maz'ya \cite{cianchi19}.

 The heuristic idea of the proof is to differentiate the equation \eqref{eq:p-parab1}, choose a test function $\vp=\abs{Du}^{s}u_{x_k}\phi^2$ and use a fundamental inequality (the name stems from \cite{dongpzz20} for a related inequality)
\begin{align*}
|Du |^4|D^2u|^2\geq2|Du|^2|D^2u Du|^2+\frac{(|Du|^2\Delta u-\Delta_\infty u)^2}{n-1}-(\Delta_\infty u)^2
\end{align*}
from \cite{sarsa20}, which holds for any smooth function. Here  $\Delta u:=\sum_{i=1}^n u_{x_ix_i}$ denotes the Laplacian, $\Delta_\infty u:=\sum_{i,j=1}^n u_{x_i x_j} u_{x_i}u_{x_j}$ the infinity Laplacian, and $|D^2u|:=(\sum_{i,j=1}^n u_{x_i x_j}^2)^{1/2}$ the Hilbert-Schmidt norm for matrices.
Surprisingly, it is sufficient for the sharp result to use the previous inequality in a rather simple form
\begin{align}
\label{eq:simple-fund-ineq}
|Du |^4|D^2u|^2\geq2|Du|^2|D^2u Du|^2 -(\Delta_\infty u)^2,
\end{align}
which we obtain by an elementary fact that the square $(|Du|^2\Delta u-\Delta_\infty u)^2$ is nonnegative. Naturally the fact that the form \eqref{eq:simple-fund-ineq} is sufficient, simplifies the proof. At the same time, the form of  \eqref{eq:simple-fund-ineq} makes the coefficient $C=C(p,s)$ of estimate \eqref{eq:QuantitativeBoundforSecondDerivatives} in Theorem \ref{thm:MainTheorem} independent of $n$. Note that in the elliptic case, the fundamental inequalities in \cite{dongpzz20,sarsa20} include the parameter $n$ so that both the range of $s$ and $C$ depend on $n$.

Unlike the second order regularity, the lower order regularity of the parabolic $p$-Laplace equation has been extensively studied since the 1980s, see DiBenedetto's monograph \cite{dibenedetto93} as well as for example \cite{bogeleindm13, dibenedettogv08, dibenedettogv12, kinnunenl00, kuusi08, urbano08, vazquez06}. In the elliptic case, the second order Sobolev regularity has been studied in addition to above mentioned \cite{dongpzz20} for example in \cite{attouchir18, bojarskii84, manfrediw88}, and for a different parabolic equation in \cite{lindqvisth20}.

\section{Preliminaries and main results}

Let $x_0\in \R^{n}$, $n\geq1$ and $r>0$. We denote by  $$B_r(x_0)=\{x \in \R^{n} : |x_0-x| <r\}$$
the usual Euclidean ball in $\R^{n}$.
For a space-time point $(x_0,t_0)\in\R^{n+1}$ and a radius $r>0$, we define the parabolic cylinder as
$$ Q_r(x_0,t_0):=B_r(x_0)\times(t_0-r^2,t_0+r^2). $$
To ease the notation, we may write $Q_r:=Q_r(x_0,t_0)$.
Let $\Om\subset\Rn$ denote an open domain. For $T>0$, we set
$$ \Om_T:=\Om \times(0,T). $$
If $U$ is compactly contained in $\Om$ i.e.\ $U\subset \Om$ and the closure of $U$ is a compact subset of  $\Om$, we write  $U \Subset \Om$ . For $0<t_1<t_2<\infty$, we set
$$ U_{t_1,t_2}:=U\times (t_1,t_2). $$

We denote the first partial derivatives of a function $u\colon\Om_T\to\R$ by $u_{x_k}$ and $u_{t}$. The spatial gradient is denoted by $Du$, and the second derivatives by $u_{x_i x_j}$. Further, $D^2 u$ stands for matrix of second derivatives with respect to the space variables.
As usual, the Sobolev space $W^{1,p}(U)$ denotes the space of measurable functions $u$ such that $u \in L^p(U)$ and the distributional first partial derivatives $u_{x_i}$ exist in $U$ and belong to $L^p(U)$. We use the norm
$$ \|u\|_{W^{1,p}(U)} = \|u\|_{L^p(U)} + \|D u\|_{L^p(U)}. $$
By the \emph{parabolic Sobolev space} $L^p(t_1,t_2;W^{1,p}(U))$, with $0<t_1<t_2<\infty$, we mean the space of measurable functions $u(x,t)$ such that the mapping $x \mapsto u(x,t)$ belongs to $W^{1,p}(U)$ for almost every $t_1 < t < t_2$ and the norm
\begin{equation*}
	\| u \|_{L^p(t_1,t_2;W^{1,p}(U))} := \biggl(\int_{t_1}^{t_2} \|u(\cdot, t)\|_{W^{1,p}(U)}^p \, dt\biggr)^{1/p}
\end{equation*}
is finite.
The space $C(\Om _T)$ denotes the space of continuous functions on $\Om _T$ and $C_0^\infty(\Om _T)$ denotes the space of smooth, compactly supported functions on $\Om_T$. A function belongs
to the local Sobolev space $W^{1,p}_{\text{loc}}(\Om)$ if it belongs to $W^{1,p}(\Om')$ for every open $\Om' \Subset \Om$. Other local spaces are defined analogously.

We study weak solutions to the parabolic $p$-Laplace equation
\begin{equation} \label{eq:p-parab}
    u_t-\Delta_p u=0 \quad\text{in }\Om _T,
\end{equation}
where
$$ \Delta_p u:=\diverg\big(|Du|^{p-2}Du\big) $$
is the $p$-Laplace operator with $1<p<\infty$.

\begin{definition}
A function $u:\Om_T \to [-\infty,\infty]$ is a
\emph{weak solution} to equation \eqref{eq:p-parab}
if whenever $U_{t_1,t_2} \Subset \Om _T$ is an open cylinder,
we have
$u \in C(U_{t_1,t_2} )\cap L^{p}(t_1,t_2;W^{1,p}(U))$, and
 $u$ satisfies the integral equality
\[
\int_{0}^{T}\int_{\Om} \abs{D u}^{p-2} \la  D u,
D\phi \,\ra  dx\,dt - \int_{0}^{T}\int_{\Om} u
\phi_t \, dx\,dt   =  0
\quad \text{for all }\phi \in C_0^\infty(\Om _T).
\]
Such solutions are called  \emph{$p$-parabolic functions}.

\end{definition}

Under the above definition, weak solutions are equivalent to viscosity solutions to \eqref{eq:p-parab} for $1<p<\infty$, see \cite{juutinenlm01,parviainenv20,siltakoski21}. In this setting, gradients are bounded and H\"{o}lder continuous by a recent work of Imbert, Jin and Silvestre \cite{imbertjs19}.  Their result covers a more general class of equations containing the parabolic $p$-Laplace equation and the normalized $p$-parabolic equation arising from the game theory \cite{manfredipr10}.
For earlier $C^{1,\alpha}$-regularity results based on the variational approach with various assumptions, see DiBenedetto and Friedman \cite{dibenedettof85}, Wiegner \cite{wiegner86}, Chen \cite{chen87}, and Chapter~IX in \cite{dibenedetto93}.

\subsection{Main results}
\quad

For $s\in\R$, we define the vector field $V_s\colon\Rn\to\Rn$ as
\begin{equation}
V_s(z):=
    \begin{cases}
    |z|^{\frac{p-2+s}{2}}z \quad &\text{for }z\in \Rn\setminus\{0\}; \\
    0 \quad&\text{for }z=0.
    \end{cases}
\end{equation}

\begin{theorem} \label{thm:MainTheorem}
Let $u\colon\Om _T\to\R$ be a weak solution to the parabolic $p$-Laplace equation \eqref{eq:p-parab}. If $s>-1$, then $ D(V_s(Du)) $ exists and belongs to $L^{2}_\loc(\Om _T)$.
Moreover, we have the estimate
\begin{align} \label{eq:QuantitativeBoundforSecondDerivatives}
    \int_{Q_r} |D(V_s(Du))|^2dxdt
    \leq \frac{C}{r^2}\Big(
    \int_{Q_{2r}} |V_s(Du)|^2dxdt
    +\int_{Q_{2r}} |Du|^{s+2}dxdt
    \Big)
\end{align}
where $C=C(p,s)>0$  and $Q_r\subset Q_{2r}\Subset\Om _T$ are concentric parabolic cylinders.
\end{theorem}

Note that here the range of $s$ and the coefficient $C$ do not depend on $n$. For the parabolic case, the range of $s$ must satisfy the constraints of both the elliptic and parabolic terms. We get the elliptic restriction $s>-1-\frac{p-1}{n-1}$ by \cite{sarsa20} and $s>-1$ rising from the parabolic terms (see Remark \ref{counterexample}). By combining them we get the restriction $s>\max\{-1-\frac{p-1}{n-1},-1\}=-1$, thus the inequality \eqref{eq:simple-fund-ineq} is sufficient and further the coefficient $C=C(p,s)$ is independent of $n$.

\begin{remark}
In particular, we may set  $s=0$, and  $s=p-2$ for any $1<p<\infty$ reproving Lindqvist's result in \cite{lindqvist08}  \,for $\abs{Du}^{\frac{p-2}{2}}Du$ and $\abs{Du}^{p-2}Du$. If $1<p<3$, we may set $s=2-p$ to reprove the second order Sobolev regularity obtained in \cite{dongpzz20}.
\end{remark}

\begin{remark}[Counterexample]\label{counterexample}
The counterexample from \cite{dongpzz20} turns out to work also in our case, and shows that  the range $s>-1$ in Theorem \ref{thm:MainTheorem} is sharp. By a direct calculation, the function
$$u(x_1,x_2)=\big(\tfrac{p}{p-1}\big)^{p-1}t+|x_1|^{1+\frac1{p-1}}$$
is a solution to \eqref{eq:p-parab} in $\mathbb{R}^2\times(0,\infty)$, and
$$|D(|Du|^{\frac{p-2+s}2}Du)|=C(p,s)|x_1|^{\frac{-p+2+s}{2(p-1)}}\in L^2_{{\rm loc}}(\mathbb{R}^2\times(0,\infty))$$
if and only if $s>-1$.

Indeed, we have
$$u_{x_1}=\tfrac{p}{p-1}|x_1|^{\frac1{p-1}-1}x_1,\quad u_{x_2}=0, $$
and
$$u_{x_1x_1}=\tfrac{p}{(p-1)^2}|x_1|^{\frac1{p-1}-1},\quad  u_{x_1x_2}=u_{x_2x_1}=u_{x_2x_2}=0. $$
Then
\begin{align*}
|D(|Du|^{\frac{p-2+s}2}Du)|
&=\Big(\frac{p}{p-1}\Big)^{\frac{p-2+s}2}|x_1|^{\frac{p-2+s}{2(p-1)}}\Big|\frac {p}{(p-1)^2}|x_1|^{\frac{2-p}{p-1}}+\frac{p-2+s}{2}\frac{p}{(p-1)^2}|x_1|^{\frac{2-p}{p-1}}\Big|\\
&=\Big(\frac{p}{p-1}\Big)^{\frac{p-2+s}2}\frac{p(p+s)}{2(p-1)^2}|x_1|^{\frac{2-p+s}{2(p-1)}}\\
&=C(p,s)|x_1|^{\frac{-p+2+s}{2(p-1)}}.
\end{align*}
\end{remark}

Once we have proven the main result, Theorem \ref{thm:MainTheorem}, the existence and integrability of the time derivative easily follows as pointed out by Lindqvist \cite{lindqvist08} and Dong, Peng, Zhang and Zhou \cite{dongpzz20}. We give the short proof for the convenience of the reader.

\begin{corollary}[Time derivative]
Let $u\colon\Om _T\to\R$ be a weak solution to the parabolic $p$-Laplace equation \eqref{eq:p-parab}. Then the time derivative $u_t$ exists as a function and
$u_t\in L^2_\loc(\Om _T)$.
\end{corollary}
\begin{proof}
Let $s=p-2>-1$, then $p+s=2(p-1)>0$ and $s+2=p>1$. By Theorem~\ref{thm:MainTheorem}, for all $ Q_r\subset Q_{2r}\Subset \Om _T$,
we have
\begin{align}\label{eqm2}
&\int_{Q_r} \big|D(\abs{Du}^{ p-2 }Du)\big|^{2} dxdt\nonumber \\
&\le \frac {C(p )}{r^2}\Big(\int_{Q_{2r}} \abs{Du}^{2(p-1)}   dxdt+  \int_{Q_{2r}} \abs{Du}^{p} dxdt\Big),
\end{align}
which implies
$$  D(\abs{Du}^{ p-2 }Du) \in L^2_{\rm loc}(\Om _T).$$
By the weak formulation
$$\int_{Q_r} u \phi_t dxdt=-\int_{Q_r} \divt(\abs{Du}^{p-2}Du) \phi dxdt \quad \text{for all } \phi\in C^\infty_0(Q_r),$$
we get that $u_t$ exists, and $u_t\in L^2_{\rm loc}(\Om _T)$.
\end{proof}

\section{Idea of the proof} \label{sec:formal}
In this section, for the convenience of the reader, we present the formal idea of the proof without excess details. In this setting, we assume that $u\in C^{\infty}(\Om _T)$ and $Du\neq0$. The detailed proof is presented in Section~\ref{sec:details}.

Differentiating with respect to $x_k$ in \eqref{eq:p-parab}, we get
\begin{align}
\label{eq:deriv-p-parab}
(u_{x_k})_t=\divt(\abs{Du}^{p-2}A\,Du_{x_k})
\end{align}
where
\begin{align*}
A=I+(p-2)\frac{Du\otimes Du}{\abs{Du}^{2}}.
\end{align*}
Here $I$ denotes the $n\times n$ identity matrix and $D u \otimes D u$ stands for the tensor product  of two vectors in $\Rn$, resulting in a matrix in $\R^{n\times n}$ with the entries $u_{x_i}u_{x_j}$.

We first study the term on the left hand side of (\ref{eq:deriv-p-parab}), and
choose a test function
\begin{align*}
\vp=\abs{Du}^{s}u_{x_k}\phi^2
\end{align*}
with $s>-1$ and $\phi \in C_0^\infty(\Om _T)$. Summing over $k$, we get
\begin{align*}
\sum_{k=1}^n\int_{\Om _T}  (u_{x_k})_t \vp dxdt&=\sum_{k=1}^n\int_{\Om _T}  (u_{x_k})_t \abs{Du}^{s}u_{x_k}\phi^2 dxdt\\
&=\sum_{k=1}^n\int_{\Om _T}  \half (u^2_{x_k})_t \abs{Du}^{s}\phi^2 dxdt\\
&=\int_{\Om _T}  \half (\abs{Du}^2)_t \abs{Du}^{s}\phi^2 dxdt\\
&=\int_{\Om _T}  \frac{1}{s+2}(\abs{Du}^{s+2})_t \phi^2 dxdt\\
&=-\int_{\Om _T}  \frac{1}{s+2}\abs{Du}^{s+2}(\phi^2)_t dxdt.
\end{align*}
Recalling (\ref{eq:deriv-p-parab}) we have
\begin{align*}
0&= \sum_{i=1}^n\int_{\Om _T}  \divt(\abs{Du}^{p-2}A\,Du_{x_i})(\abs{Du}^{s}u_{x_i}\phi^2) dxdt +\frac{2}{s+2}  \int_{\Om _T}\abs{Du}^{s+2}\phi\phi_t  \,dxdt.
\end{align*}
Now the first integral on the right hand side is of the same form as in the elliptic case, and thus the proof of \cite[Lemma 3.3]{sarsa20} gives that for any $\eta>0$,
\begin{align}
\label{eq:main-ineq}
\int_{\Om _T}& \abs{Du}^{p-2+s}\Big\{  \abs{D^{2}u}^2+(p-2+s-\eta)\abs{D\abs{Du}}^{2}+(s(p-2)-\eta)(\Delta_{\infty}^N u)^2  \Big\}\phi^{2} \ud xdt\nonumber\\
&\le \frac{C(p)}{\eta }\int_{\Om _T}\abs{Du}^{p+s} \abs{D\phi}^{2} \ud xdt+\frac{2}{s+2} \int_{\Om _T}\abs{Du}^{s+2}|\phi||\phi_t| \ud xdt,
\end{align}
where $\Delta_{\infty}^N u:=\abs{Du}^{-2}\sum_{i,j=1}^n u_{x_i x_j} u_{x_i}u_{x_j}$ stands for the normalized or game theoretic infinity Laplacian.
Observe that on the right hand side, we have bounded terms only.
As a corollary, similarly as in \cite[Corollary 3.4]{sarsa20}, we get
\begin{align}
\label{eq:cor-main-ineq}
\int_{\Om _T}& \abs{Du}^{p-2+s}   \abs{D^{2}u}^2\phi^{2} \ud xdt \nonumber \\
& \le C(p,s,\eta)\int_{\Om _T} \abs{Du}^{p-2+s} \abs{D\abs{Du}}^{2}\phi^{2} \ud xdt\nonumber \\
&+\frac{C(p)}{\eta }\int_{\Om _T}\abs{Du}^{p+s} \abs{D\phi}^{2} \ud xdt+\frac{2}{s+2} \int_{\Om _T}\abs{Du}^{s+2}|\phi||\phi_t| \ud xdt.
\end{align}
Next we estimate the first term on the right hand side in \eqref{eq:cor-main-ineq}.
Using the inequality of \cite[Corollary 2.2]{sarsa20}:
$$|Du|^4|D^2u|^2\geq2|Du|^2|D^2uDu|^2+\frac{(|Du|^2\Delta u-\Delta_\infty u)^2}{n-1}-(\Delta_\infty u)^2,$$
dividing both sides by $|Du|^4$, we have
\begin{align}\label{eq:smo-tri-ineq}
 |D^2u|^2
 &\geq2 \abs{D\abs{Du}}^{2}+\frac{(\Delta u-\Delta_\infty^N u)^2}{n-1}-(\Delta_\infty^N u)^2\nonumber\\
 &\geq2 \abs{D\abs{Du}}^{2} -(\Delta_\infty^N u)^2.
 \end{align}
On the last line we used $(\Delta u-\Delta_\infty^N u)^2\geq0$. Now we use the previous inequality in \eqref{eq:main-ineq} for the term containing $\abs{D^{2}u}^2$,
set $\eta=\min\{ \frac14(p+s),\frac16(p-1)(s+1)\}$,
 and obtain
 \begin{align}\label{eq:smo-esti}
& \abs{D^{2}u}^2+(p-2+s-\eta)\abs{D\abs{Du}}^{2}+(s(p-2)-\eta)(\Delta_{\infty}^N u)^2\nonumber\\
&\geq2 \abs{D\abs{Du}}^{2} -(\Delta_{\infty}^{N}u)^2+(p-2+s-\eta )\abs{D\abs{Du}}^{2}+(s(p-2)-\eta)(\Delta_{\infty}^N u)^2\nonumber\\
&= (p+s-\eta )\abs{D\abs{Du}}^{2}+(s(p-2)-1-\eta)(\Delta_{\infty}^N u)^2\nonumber\\
&= \eta  \abs{D\abs{Du}}^{2}+(p+s-2\eta )\abs{D\abs{Du}}^{2}+(s(p-2)-1-\eta)(\Delta_{\infty}^N u)^2\nonumber\\
&\geq \eta  \abs{D\abs{Du}}^{2}+(p+s+s(p-2)-1-3\eta)(\Delta_{\infty}^N u)^2\nonumber\\
&=\eta  \abs{D\abs{Du}}^{2}+( (p-1)(s+1)-3\eta)(\Delta_{\infty}^N u)^2\nonumber\\
&\geq\eta\abs{D\abs{Du}}^{2}=C(p,s)\abs{D\abs{Du}}^{2},
\end{align}
whenever $s>-1$.  We also used
\begin{align}\label{eq:trivialineq}
\abs{D\abs{Du}}^2=\frac{\abs{D^{2}uDu}^2}{\abs{Du}^2}\ge\Bigg(\frac{\la D^2 uDu,Du\ra}{\abs{Du}^2}\Bigg)^2= (\Delta_{\infty}^{N}u)^2.
\end{align}
Thus
\begin{align*}
\int_{\Om _T}& \abs{Du}^{p-2+s}
 \abs{D\abs{Du}}^{2}\phi^{2} \ud xdt\\
&\le  C(p,s)\Big( \int_{\Om _T}\abs{Du}^{p+s} \abs{D\phi}^{2} \ud xdt+ \int_{\Om _T}\abs{Du}^{s+2}|\phi||\phi_t| \ud xdt\Big).
\end{align*}
 Combining this with $(\ref{eq:cor-main-ineq})$, we get
\begin{align*}
\int_{\Om _T}\abs{Du}^{p-2+s}   \abs{D^{2}u}^2\phi^{2} \ud xdt
&\le C(p,s)\Big(\int_{\Om _T}\abs{Du}^{p+s} \abs{D\phi}^{2} \ud xdt+  \int_{\Om _T}\abs{Du}^{s+2}|\phi||\phi_t|\ud xdt\Big).
\end{align*}
By a direct calculation
\begin{align}\label{eq1}
\int_{\Om _T}\big| D(\abs{Du}^{\frac{p-2+s}2}Du)\big|^{2}\phi^2\ud xdt
\leq C(p,s)\int_{\Om _T}\abs{Du}^{p-2+s}|D^2u|^2\phi^2\ud xdt,
\end{align}
and combining this with the previous estimate, we finally get
\begin{align}\label{eqm1}
&\int_{\Om _T}\big|D(\abs{Du}^{\frac{p-2+s}2}Du)\big|^{2}\phi^2\ud xdt \nonumber\\
&\le C(p,s)\Big(\int_{\Om _T}\abs{Du}^{p+s} \abs{D\phi}^{2} \ud xdt+  \int_{\Om _T}\abs{Du}^{s+2}|\phi||\phi_t| \ud xdt\Big).
\end{align}
The estimate in Theorem \ref{thm:MainTheorem} is obtained by choosing $\phi$ as a standard cutoff function.

\section{Detailed proof}
\label{sec:details}
In this section, we present a detailed proof of  Theorem \ref{thm:MainTheorem} by regularizing the equation \eqref{eq:p-parab}. Solutions to the regularized equation will be smooth, and thus the differentiation of this equation is justified. Since the obtained estimates will be uniform with respect to the regularization, we will be able to pass to the original equation at the end.

To start with the above plan, let $u\colon\Om_T\to\R$ be a $p$-parabolic function. Fix a smooth subdomain $U\Subset\Om$ and $0<t_1<t_2<\infty$ such that $U_{t_1,t_2}\Subset \Om_T$. Let $\epsilon>0$ be small and $\ue\colon U_{t_1,t_2}\to\R$ be a weak solution to
\begin{equation} \label{eq:RegularizedProblem}
    \begin{cases}
    \begin{aligned}
    \ue_t-\diverg\big(\mu^{p-2}D\ue\big)=0 &\quad\text{in } U_{t_1,t_2}; \\
    \ue=u &\quad\text{on }\partial_p U_{t_1,t_2},
    \end{aligned}
    \end{cases}
\end{equation}
where
$$ \mu:=\sqrt{|D\ue|^2+\epsilon} $$
and the \emph{parabolic boundary} is defined as
\[
\partial_p U_{t_1,t_2}=(\overline U\times \{t_1\})\cup(
\partial U\times (t_1,t_2]). \]
According to standard parabolic theory, we get $\ue \in C^\infty( U_{t_1,t_2})\cap C( \overline{U} _{t_1,t_2})$, see \cite{dibenedettof85,wiegner86}.

\begin{lemma} \label{lem:EstimateforRegularization}
Let $\ue\colon  U_{t_1,t_2}\to\R$ be a weak solution to \eqref{eq:RegularizedProblem}. If $s>-1$, then for any $\phi\in C^{\infty}_0( U_{t_1,t_2})$, we have
\begin{align*}
    &\int_{ U_{t_1,t_2}} \mu^{p-2+s}|D^2\ue|^2\phi^2dxdt \\
    &\leq C\Big(
    \int_{ U_{t_1,t_2}} \mu^{p-2+s}|D\ue|^2|D\phi|^2dxdt
    +\int_{ U_{t_1,t_2}} \mu^{s+2}|\phi||\phi_t|dxdt
    \Big)
\end{align*}
where $C=C(p,s)>0$ is independent of $\epsilon$.
\end{lemma}

To prove Lemma \ref{lem:EstimateforRegularization}, we use the inequality \eqref{eq:simple-fund-ineq}. In Section~\ref{sec:formal}, under the assumption $  Du \neq0$, we can directly divide both sides of \eqref{eq:simple-fund-ineq} by $|Du|^4$ to get the inequality \eqref{eq:smo-tri-ineq}, which gives the lower bound of $\abs{D^2u}^2$. In order to get an inequality similar to \eqref{eq:smo-tri-ineq}, we also need to consider the case when $Du=0$.
Thus we reformulate \eqref{eq:simple-fund-ineq} here in a way that allows us to apply it in this context.

For the reformulation, we introduce some notations.
Let $v\colon U_{t_1,t_2}\to\R$ be a smooth function. In particular, $|Dv|$ is locally Lipschitz continuous (by triangle inequality) and thus, by Rademacher's theorem, differentiable almost everywhere on each time slice, hence also in $U_{t_1,t_2}$.

Note that if $(x_0,t_0)\in U_{t_1,t_2}$ is a space-time point where $|Dv|$ is differentiable and $Dv(x_0,t_0)=0$, then $D|Dv|(x_0,t_0)=0$. Indeed, if we had $D|Dv|(x_0,t_0)\neq 0$, then we could find a point $\xi\in U\times\{t_0\}$ (close to $(x_0,t_0)$) such that $|Dv|(\xi)<0$, which is obviously impossible.
On the other hand, if $Dv(x_0,t_0)\neq 0$ for some $(x_0,t_0)\in U_{t_1,t_2}$, then $|Dv|$ is differentiable at $(x_0,t_0)$ and
$$ D|Dv|(x_0,t_0)=\frac{D^2v(x_0,t_0)Dv(x_0,t_0)}{|Dv(x_0,t_0)|}. $$

For each point in $U_{t_1,t_2}$ where $Dv\neq0$, we fix an orthonormal basis of $\Rn$, $\{e_1,\ldots,e_n\}$,
such that $e_n=\frac{Dv}{|Dv|}$. Hence we have, for those points where $Dv\neq0$,
$$ \frac{D^2vDv}{|Dv|}
=\la e_1,D|Dv|\ra e_1
+\ldots
+\la e_{n-1},D|Dv|\ra e_{n-1}
+\la \frac{Dv}{|Dv|},D|Dv|\ra\frac{Dv}{|Dv|}. $$
For those points where $|Dv|$ is differentiable, let us define the part of $D|Dv|$ which is tangential to the spatial level sets of $v$ as
\begin{align*}
    D_T|Dv|:=
    \begin{cases}
    \la e_1,D|Dv|\ra e_1
    +\ldots
    +\la e_{n-1},D|Dv|\ra e_{n-1} &\quad\text{if }Dv\neq 0, \\
    0 &\quad\text{if }Dv= 0,
    \end{cases}
\end{align*}
and its orthogonal counterpart, the normalized infinity Laplacian, as
\begin{align*}
    \ilN v:=
    \begin{cases}
    \la \frac{Dv}{|Dv|},D|Dv|\ra=\frac{\il v}{|Dv|^2} &\quad\text{if }Dv\neq 0, \\
    0 &\quad\text{if }Dv= 0.
    \end{cases}
\end{align*}
We employ these notations to write
\begin{equation} \label{eq:OrthogonalRepresentation}
    |D|Dv||^2=|D_T|Dv||^2+(\ilN v)^2 \quad\text{a.e. in }U_{t_1,t_2}.
\end{equation}

In Section \ref{sec:formal}, without dividing $|D|Dv||^2$ into two parts, we use the inequality \eqref{eq:smo-tri-ineq} and \eqref{eq:trivialineq} to get the estimate \eqref{eq:smo-esti}. When using \eqref{eq:trivialineq}, we need to be careful and check that if the coefficient of $|D|Dv||^2$ is nonnegative. For the regularization, the coefficients of each terms become more complicated, thus by using the equality \eqref{eq:OrthogonalRepresentation}, we can consider the coefficients together in the last step of the estimate. Now we can restate \eqref{eq:simple-fund-ineq}.

\begin{lemma} \label{lem:TrivialInequality}
Let $v\colon U_{t_1,t_2}\to\R$ be a smooth function. Then
\begin{align}\label{eq:tri-ineq}
    |D^2v |^2\geq 2|D_T|Dv| |^2+(\ilN v )^2\quad\text{  a.e. in } U_{t_1,t_2}.
\end{align}
\end{lemma}

\begin{proof}
Recall that $|Dv|$ is differentiable a.e. in $U_{t_1,t_2}$. From now on, consider such points of $U_{t_1,t_2}$ where $|Dv|$ is differentiable.
When $n=1$, by the definition of $D_T|D v|$, we have $D_T|D v|=0$ and \eqref{eq:tri-ineq} is obviously an identity. Then we consider the case $n\geq2$.
If $Dv=0$, then \eqref{eq:tri-ineq} holds trivially by what we defined above.
If $Dv\neq 0$, then by \cite[Corollary 2.2]{sarsa20}, we have
$$|D v|^4|D^2v|^2\geq2|Dv|^2|D^2v Dv|^2+\frac{(|Dv|^2\Delta v-\Delta_\infty v)^2}{n-1}-(\Delta_\infty v)^2.$$
Dividing both sides by $\abs{Dv}^4$, using the definitions of $D_T|D v|$ and $\ilN v$, we get the desired inequality by following:
\begin{align*}
\abs{D^2v}^2&\geq 2\abs{D\abs{Dv}}^2+\frac{(\Delta v-\ilN v)^2}{n-1}-(\ilN v)^2\\
&\geq 2\abs{D\abs{Dv}}^2 -(\ilN v)^2\\
&=2|D_T|D v||^2+(\ilN v)^2.\qedhere
\end{align*}
\end{proof}

\begin{proof}[Proof of Lemma \ref{lem:EstimateforRegularization}]
The spatial partial derivatives $\ue_{x_k}$, $k=1,\ldots,n$, solve
\begin{equation} \label{eq:RegularizedLinearization}
    (\ue_{x_k})_t-\diverg\big(\mu^{p-2}AD\ue_{x_k}\big)=0
\end{equation}
where
$$ A=I+(p-2)\frac{D\ue\otimes D\ue}{\mu^2}. $$
Note that
\begin{equation} \label{eq:EllipticityofA}
\min \{1, p-1\}I \leq A \leq \max \{1,p-1\} I
\end{equation}
uniformly in $ U_{t_1,t_2}$ and for $\epsilon$.

We multiply the equation \eqref{eq:RegularizedLinearization} with $\mu^s\ue_{x_k}$, where $s>-1$, and obtain
\begin{equation} \label{eq:TestedLinearization}
    \mu^s\ue_{x_k}(\ue_{x_k})_t
    -\mu^s\ue_{x_k}\diverg\big(\mu^{p-2}A\,D\ue_{x_k}\big)=0.
\end{equation}
For the first item in the above display we note that
\begin{align} \label{eq:TimeDerivativeTrick}
    \ue_{x_k}(\ue_{x_k})_t
    =
    \frac{1}{2}\big((\ue_{x_k})^2+\tfrac{\epsilon}{n}\big)_t.
\end{align}
Summing \eqref{eq:TestedLinearization} over $k=1,\ldots,n$ and taking \eqref{eq:TimeDerivativeTrick} into account gives that
\begin{align} \label{eq:LinearizationsSummed}
    \frac{1}{s+2}(\mu^{s+2})_t
    -\mu^s\sum_{k=1}^n\ue_{x_k}\diverg\big(\mu^{p-2}A\,D\ue_{x_k}\big) =0.
\end{align}
Observe that
\begin{equation}
\begin{aligned}
    \diverg\big(\mu^{p-2+s}A\,D^2\ue D\ue\big)
    &=
    \sum_{k=1}^n\diverg\big((\mu^s \ue_{x_k})(\mu^{p-2}A\,D\ue_{x_k})\big) \\
    &=
    \mu^s\sum_{k=1}^n \ue_{x_k}\diverg\big(\mu^{p-2}A\,D \ue_{x_k}\big)
    +\mu^{p-2+s}\Big(|D^2\ue|^2  \\
    &\quad
    +(p-2+s)\frac{|D^2\ue D\ue|^2}{\mu^2}
    +s(p-2)\frac{(\il\ue)^2}{\mu^4}\Big).
\end{aligned}
\end{equation}
Above we used
\begin{align*}
\la A\,D \ue_{x_k},D \ue_{x_k}\ra&=\abs{D \ue_{x_k}}^2+(p-2)\frac{\la D \ue,D \ue_{x_k}\ra^2 }{\mu^2}, \\
\la A\, D \ue_{x_k}, D^2 \ue D \ue\ra &=\la D \ue_{x_k} ,D^2 \ue D \ue \ra+(p-2)\frac{\la D \ue,D \ue_{x_k}\ra \Delta_{\infty}\ue}{\mu^2},
\end{align*}
and a straightforward computation.
In other words,
\begin{equation} \label{eq:DivergenceIdentity}
\begin{aligned}
    \mu^s\sum_{k=1}^n \ue_{x_k}\diverg\big(\mu^{p-2}A\,D \ue_{x_k}\big)
    =
    \diverg\big(\mu^{p-2+s}A\,D^2\ue D\ue\big)
    -\mu^{p-2+s}\sigma,
\end{aligned}
\end{equation}
where
\begin{align*}
    \sigma:
    =
    |D^2\ue|^2+(p-2+s)\frac{|D^2\ue D\ue|^2}{\mu^2}+s(p-2)\frac{(\il\ue)^2}{\mu^4}.
\end{align*}
By \eqref{eq:LinearizationsSummed} and \eqref{eq:DivergenceIdentity}, we have
\begin{align} \label{eq:BasicIdentity}
    \mu^{p-2+s}\sigma
    =
    \diverg\big(\mu^{p-2+s}A\,D^2\ue D\ue\big)-\frac{1}{s+2}(\mu^{s+2})_t.
\end{align}
We claim that for $s>-1$, we can find a small number $\lambda=\lambda( p,s)>0$ such that
\begin{equation} \label{eq:LowerBoundforsigma}
    \lambda|D^2\ue|^2\leq \sigma \quad\text{a.e. in } U_{t_1,t_2}.
\end{equation}
Observe that this is not a trivial inequality since not all the coefficients are positive in $\sigma$.

If \eqref{eq:LowerBoundforsigma} holds, then the desired estimate follows easily. Indeed, we plug the estimate \eqref{eq:LowerBoundforsigma}  into the equation \eqref{eq:BasicIdentity} to obtain
\begin{align} \label{eq:DesiredPointwiseEstimate}
    \lambda\mu^{p-2+s}&|D^2\ue|^2
    \leq
    \diverg\big(\mu^{p-2+s}A\,D^2\ue D\ue\big)-\frac{1}{s+2}(\mu^{s+2})_t.
\end{align}
Let $\phi\in C^{\infty}_0( U_{t_1,t_2})$. Multiplying \eqref{eq:DesiredPointwiseEstimate} by $\phi^2$ and then integrating over $U_{t_1,t_2}$ yields
\begin{align*}
    \lambda\int_{ U_{t_1,t_2}}\mu^{p-2+s}&|D^2\ue|^2\phi^2dxdt\\
    &\leq
    \int_{ U_{t_1,t_2}}\Big(\diverg\big(\mu^{p-2+s}A\,D^2\ue D\ue\big)-\frac{1}{s+2}(\mu^{s+2})_t\Big)\phi^2dxdt.
\end{align*}
We employ integration by parts, \eqref{eq:EllipticityofA} and Young's inequality to obtain the upper bound of the right hand side term in above inequality,
\begin{align*}
    &\int_{ U_{t_1,t_2}}\Big(\diverg\big(\mu^{p-2+s}A\,D^2\ue D\ue\big)-\frac{1}{s+2}(\mu^{s+2})_t\Big)\phi^2dxdt \\
    &=
    -\int_{ U_{t_1,t_2}}\mu^{p-2+s} \la A\,D^2\ue D\ue , D\phi^2\ra dxdt +\frac{1}{s+2}\int_{ U_{t_1,t_2}}\mu^{s+2}(\phi^2)_tdxdt \\
    &\leq
    \eta\int_{ U_{t_1,t_2}}\mu^{p-2+s}|D^2\ue|^2\phi^2dxdt
    +\frac{C}{\eta}\int_{ U_{t_1,t_2}}\mu^{p-2+s}|D\ue|^2|D\phi|^2dxdt \\
    &\quad
    +\frac{2}{s+2}\int_{ U_{t_1,t_2}}\mu^{s+2}|\phi||\phi_t|dxdt
\end{align*}
for any $\eta>0$ and some constant $C=C(p)>0$. The desired estimate follows by choosing $\eta=\frac{\lambda}{2}$.

It remains to prove \eqref{eq:LowerBoundforsigma}.
As explained above in this section,
we can write
\begin{align*}
    \sigma
    &=
    |D^2\ue|^2+(p-2+s)\frac{|D\ue|^2}{\mu^2}|D|D\ue||^2+s(p-2)\frac{|D\ue|^4}{\mu^4}(\ilN\ue)^2 \\
    &=
    |D^2\ue|^2+(p-2+s)\frac{|D\ue|^2}{\mu^2}|D_T|D\ue||^2 \\
    &\quad
    +\Big((p-2+s)\frac{|D\ue|^2}{\mu^2}+s(p-2)\frac{|D\ue|^4}{\mu^4}\Big)(\ilN\ue)^2
\end{align*}
almost everywhere in $ U_{t_1,t_2}$. For $\lambda\in(0,1)$, we write
\begin{equation} \label{eq:Divisionofsigma}
   \sigma = \lambda \sigma+(1-\lambda)\sigma.
\end{equation}
For the latter part of $\sigma$ on the right hand side of \eqref{eq:Divisionofsigma},
we utilize the nonnegativity of the $|D^2\ue|^2$-term via the inequality of Lemma \ref{lem:TrivialInequality}:
$$ |D^2\ue|^2\geq 2|D_T|D\ue||^2+(\ilN\ue)^2. $$
We obtain a lower bound
\begin{align*}
    \sigma
    &\geq
    \Big(2+(p-2+s)\frac{|D\ue|^2}{\mu^2}\Big)|D_T|D\ue||^2 \\
    &\quad
    +\Big(1+(p-2+s)\frac{|D\ue|^2}{\mu^2}+s(p-2)\frac{|D\ue|^4}{\mu^4}\Big)(\ilN\ue)^2=:\tau.
\end{align*}
Now we have
\begin{equation} \label{eq:AbstractLowerBoundforsigma}
    \sigma\geq \lambda(\sigma-\tau)+\tau.
\end{equation}
Writing
$$ 1=\frac{|D\ue|^2}{\mu^2}+\frac{\epsilon}{\mu^2} $$
allows us to divide the terms in $\sigma$ and $\tau$ according to the degree of $\epsilon$.
This kind of regrouping is useful, because it separates the main terms that appear also in the formal calculation of Section \ref{sec:formal} from those terms that appear as a result of the regularization.

Indeed, we write
\begin{equation} \label{eq:tau}
\begin{aligned}
    \tau
    &=
    \bigg(2\Big(\frac{|D\ue|^2}{\mu^2}+\frac{\epsilon}{\mu^2}\Big)^2
    +(p-2+s)\frac{|D\ue|^2}{\mu^2}\Big(\frac{|D\ue|^2}{\mu^2}+\frac{\epsilon}{\mu^2}\Big)\bigg)|D_T|D\ue||^2 \\
    &\quad
    +\bigg(\Big(\frac{|D\ue|^2}{\mu^2}+\frac{\epsilon}{\mu^2}\Big)^2
    +(p-2+s)\frac{|D\ue|^2}{\mu^2}\Big(\frac{|D\ue|^2}{\mu^2} +\frac{\epsilon}{\mu^2}\Big) \\
    &\quad
    +s(p-2)\frac{|D\ue|^4}{\mu^4}\bigg)(\ilN\ue)^2 \\
    &=
    \Big((p+s)\frac{|D\ue|^4}{\mu^4}+(p+s+2)\frac{\epsilon|D\ue|^2}{\mu^4}+\frac{2\epsilon^2}{\mu^4}\Big)|D_T|D\ue||^2 \\
    &\quad
    +\Big((p-1)(s+1)\frac{|D\ue|^4}{\mu^4}+(p+s)\frac{\epsilon|D\ue|^2}{\mu^4}+\frac{\epsilon^2}{\mu^4}\Big)(\ilN\ue)^2
\end{aligned}
\end{equation}
and
\begin{equation} \label{eq:sigmaminustau}
\begin{aligned}
\sigma-\tau
    &=
    |D^2\ue|^2-2|D_T|D\ue||^2-(\ilN\ue)^2 \\
    &=
    |D^2\ue|^2
    -2\Big(\frac{|D\ue|^4}{\mu^4}+\frac{2\epsilon|D\ue|^2}{\mu^2}+\frac{\epsilon^2}{\mu^4}\Big)|D_T|D\ue||^2 \\
    &\quad
    -\Big(\frac{|D\ue|^4}{\mu^4}+\frac{2\epsilon|D\ue|^2}{\mu^2}+\frac{\epsilon^2}{\mu^4}\Big)(\ilN\ue)^2.
\end{aligned}
\end{equation}
As we plug \eqref{eq:tau} and \eqref{eq:sigmaminustau} into \eqref{eq:AbstractLowerBoundforsigma}, we can easily choose $\lambda=\lambda(p,s)>0$ so small that
\begin{align*}
    \sigma
    &\geq
    \lambda|D^2\ue|^2
    +\Big((p+s-2\lambda)\frac{|D\ue|^4}{\mu^4}
    +(p+s+2-4\lambda)\frac{\epsilon|D\ue|^2}{\mu^4} \\
    &\quad
    +(2-2\lambda)\frac{\epsilon^2}{\mu^4}\Big)|D_T|D\ue||^2
    +\Big(((p-1)(s+1)-\lambda)\frac{|D\ue|^4}{\mu^4} \\
    &\quad
    +(p+s-2\lambda)\frac{\epsilon|D\ue|^2}{\mu^4}
    +(1-\lambda)\frac{\epsilon^2}{\mu^4}\Big)(\ilN\ue)^2 \\
    &\geq
    \lambda|D^2\ue|^2.
\end{align*}
This is indeed possible because $s>-1$.
Now that we have shown \eqref{eq:LowerBoundforsigma}, the proof is finished.
\end{proof}

\begin{proof} [Proof of Theorem \ref{thm:MainTheorem}]
To prove Theorem \ref{thm:MainTheorem}, we need to justify letting $\epsilon\to0$ in Lemma \ref{lem:EstimateforRegularization}.
For notational convenience, we introduce the regularized version of the vector field $V_s$ which corresponds to Lemma  \ref{lem:EstimateforRegularization}. Let us define $\Ve_s\colon\Rn\to\Rn$ as
$$ \Ve_s(z):=(|z|^2+\epsilon)^{\frac{p-2+s}{4}}z \quad\text{for }z\in\Rn. $$
Similarly to \eqref{eq1}, by Lemma \ref{lem:EstimateforRegularization}, there exists a constant $C=C( p,s)>0$ such that
\begin{equation} \label{eq:RegularizedQuantitativeBoundforSecondDerivatives}
\begin{aligned}
    \int_{ U_{t_1,t_2}}|D(\Ve_s(D\ue))|^2\phi^2dxdt
    &\leq
    C\Big(\int_{ U_{t_1,t_2}}|\Ve_s(D\ue)|^2|D\phi|^2dxdt \\
    &\quad
    +\int_{ U_{t_1,t_2}}(|D\ue|^2+\epsilon)^{\frac{s+2}{2}}|\phi||\phi_t|dxdt \Big)
\end{aligned}
\end{equation}
for any $\phi\in C^{\infty}_0( U_{t_1,t_2})$.

The estimate \eqref{eq:QuantitativeBoundforSecondDerivatives} can be derived from \eqref{eq:RegularizedQuantitativeBoundforSecondDerivatives} as follows.
Let us fix a space-time point $(x_0,t_0)\in U_{t_1,t_2}$. Let $r>0$ be small enough such that the parabolic cylinder with center $(x_0,t_0)$ and radius $2r$ fits inside $ U_{t_1,t_2}$, that is $Q_{2r}\Subset U_{t_1,t_2}$. Let $\phi\in C^{\infty}_0( U_{t_1,t_2})$ be a cutoff function such that
\begin{align}\label{ConditionsForTestFunctions} \phi\equiv1\;\;\text{in }Q_r,
\quad |\phi|\leq 1,
\quad \spt \phi \subset Q_{2r},
\quad |D\phi|\leq \frac{10}{r}
\quad{\text{and}}\quad |\phi_t|\leq \frac{10}{r^2}.
\end{align}
The estimate \eqref{eq:RegularizedQuantitativeBoundforSecondDerivatives} implies that
\begin{equation} \label{eq:RegularizedQuantitativeBoundforSecondDerivativesinBall}
\int_{Q_r}|D(\Ve_s(D\ue))|^2dxdt
\leq \frac{C}{r^2}\Big(\int_{Q_{2r}}|\Ve_s(D\ue)|^2dxdt+\int_{Q_{2r}}(|D\ue|^2+\epsilon)^{\frac{s+2}{2}}dxdt\Big)
\end{equation}
for $C=C( p,s)>0$.

Since $s>-1$, we can apply for example \cite{imbertjs19} to conclude for the gradient $$\norm{D\ue}_{C^{\alpha}(Q_{{2r}})}\le C.$$
Thus  $D\ue$ converge uniformly (and strongly in $L^p$) by Arzel\`a-Ascoli theorem. It follows that the limit $u$ is a solution to (\ref{eq:p-parab}). Moreover, the right hand side of \eqref{eq:RegularizedQuantitativeBoundforSecondDerivativesinBall} is thus bounded from above by a constant independent of $\epsilon$. Thus $\{D(\Ve_s(D\ue))\}_{\epsilon}$ is bounded in $L^2(Q_{r})$, and consequently we may extract a subsequence that converges weakly in $L^2(Q_r)$. Further, using integration by parts, we see that the limit is $D(V_s(Du))$, and thus
\begin{align*}
 \int_{Q_{ r}}|D(V_s(Du))|^2dxdt&\le \liminf_{\eps\to 0} \int_{Q_{ r}}|D(\Ve_s(D\ue))|^2dxdt\\
 &\le \lim_{\eps\to 0} \frac{C}{r^2}\Big(\int_{Q_{2r}}|\Ve_s(D\ue)|^2dxdt+\int_{Q_{2r}}(|D\ue|^2+\epsilon)^{\frac{s+2}{2}}dxdt\Big)\\
 &\le \frac{C}{r^2}\Big(\int_{Q_{2r}}|V_s(Du)|^2dxdt+\int_{Q_{2r}}|Du|^ {s+2} dxdt\Big),
\end{align*}
which is the desired estimate.
\end{proof}

\section*{Acknowledgement}
The first author was supported by China Scholarship Council, no. 202006020186. The third author was supported  by  the Academy of Finland, the Centre of Excellence in Analysis and Dynamics Research and the Academy of Finland, project 308759.

\def\cprime{$'$} \def\cprime{$'$} \def\cprime{$'$}


\begin{thebibliography}{10}

\bibitem{attouchir18}
A.~Attouchi and E.~Ruosteenoja.
\newblock Remarks on regularity for $p$-Laplacian type equations in non-divergence form.
\newblock {\em J. Differential Equations}, 265(5):1922--1961,2018.

\bibitem{bogeleindm13}
V.~B\"{o}gelein, F.~Duzaar, and G.~Mingione.
\newblock The regularity of general parabolic systems with degenerate
  diffusion.
\newblock {\em Mem. Amer. Math. Soc.}, 221(1041):vi+143, 2013.


\bibitem{bojarskii84}
B.~Bojarski and T.~Iwaniec.
\newblock $p$-harmonic equation and quasiregular mappings.
\newblock {\em Partial differential equations} (Warsaw, 1984), {\em Banach Center Publ.}, 19, PWN,
Warsaw: 25--38, 1987.

\bibitem{chen87}
Y. Chen.
\newblock H\"{o}lder continuity of the gradient of the solutions of certain
  degenerate parabolic equations.
\newblock {\em Chinese Ann. Math. Ser. B}, 8(3):343--356, 1987.
\newblock A Chinese summary appears in Chinese Ann. Math. Ser. A {{\bf{8}}}
  (1987), no. 3, 534.

\bibitem{cianchi19}
A.~Cianchi and V.G.~Maz'ya.
\newblock Second-order regularity for parabolic $p$-Laplace problems.
\newblock {\em J. Geom. Anal.}, 30(2):1565--1583, 2020.

\bibitem{dibenedetto93}
E.~DiBenedetto.
\newblock {\em Degenerate parabolic equations}.
\newblock Universitext. Springer-Verlag, New York, 1993.


\bibitem{dibenedettof85}
E.~DiBenedetto and A.~Friedman.
\newblock H\"{o}lder estimates for nonlinear degenerate parabolic systems.
\newblock {\em J. Reine Angew. Math.}, 357:1--22, 1985.


\bibitem{dibenedettogv08}
E.~DiBenedetto, U.~Gianazza, and V.~Vespri.
\newblock Harnack estimates for quasi-linear degenerate parabolic differential
  equations.
\newblock {\em Acta Math.}, 200(2):181--209, 2008.

\bibitem{dibenedettogv12}
E.~DiBenedetto, U.~Gianazza, and V.~Vespri.
\newblock {\em Harnack's inequality for degenerate and singular parabolic
  equations}.
\newblock Springer Monographs in Mathematics. Springer, New York, 2012.


\bibitem{dongpzz20}
H.~Dong, F.~Peng, Y.~Zhang, and Y.~Zhou.
\newblock Hessian estimates for equations involving {$p$}-{L}aplacian via a
  fundamental inequality.
\newblock {\em Adv. Math.}, 370:107212, 40, 2020.

\bibitem{lindqvisth20}
F.A.~H${\o}$eg and P.~Lindqvist.
\newblock Regularity of solutions of the parabolic normalized $p$-Laplace equation.
\newblock {\em Adv. Nonlinear Anal.}, 9(1):7-15, 2020.



\bibitem{imbertjs19}
C.~Imbert, T.~Jin, and L.~Silvestre.
\newblock H\"{o}lder gradient estimates for a class of singular or degenerate
  parabolic equations.
\newblock {\em Adv. Nonlinear Anal.}, 8(1):845--867, 2019.

\bibitem{juutinenlm01}
P.~Juutinen, P.~Lindqvist, and J.J. Manfredi.
\newblock On the equivalence of viscosity solutions and weak solutions for a
  quasi-linear equation.
\newblock {\em SIAM J. Math. Anal.}, 33(3):699--717, 2001.

\bibitem{kinnunenl00}
J.~Kinnunen and J.L. Lewis.
\newblock Higher integrability for parabolic systems of {$p$}-{L}aplacian type.
\newblock {\em Duke Math. J.}, 102(2):253--271, 2000.

\bibitem{kuusi08}
T.~Kuusi.
\newblock Harnack estimates for weak supersolutions to nonlinear degenerate
  parabolic equations.
\newblock {\em Ann. Sc. Norm. Super. Pisa Cl. Sci. (5)}, 7(4):673--716, 2008.

\bibitem{lindqvist08}
P.~Lindqvist.
\newblock On the time derivative in a quasilinear equation.
\newblock {\em Skr. K. Nor. Vidensk. Selsk.}, (2):1--7, 2008.

\bibitem{manfredipr10}
J.J. Manfredi, M.~Parviainen, and J.D. Rossi.
\newblock An asymptotic mean value characterization for $p$-harmonic functions.
\newblock {\em Proc. Amer. Math. Soc.}, 138:881--889, 2010.

\bibitem{manfrediw88}
J.J.~Manfredi and A.~Weitsman.
\newblock On the Fatou theorem for $p$-harmonic functions.
\newblock {\em Comm. Partial Differential Equations}, 13(6):651--668, 1988.


\bibitem{parviainenv20}
M.~Parviainen and J.L.~V{\'a}zquez.
\newblock Equivalence between radial solutions of different parabolic
  gradient-diffusion equations and applications.
\newblock {\em Ann. Sc. Norm. Super. Pisa Cl. Sci. (5)}, 21:303--359, 2020.

\bibitem{sarsa20}
S.~Sarsa.
\newblock Note on an elementary inequality and its application to the
  regularity of $ p $-harmonic functions.
\newblock {\em arXiv preprint arXiv:2009.10102}, 2020.

\bibitem{siltakoski21}
J.~Siltakoski.
\newblock Equivalence of viscosity and weak solutions for a $p$-parabolic
  equation.
\newblock {\em J. Evol. Equ.}, 21(2):2047--2080, 2021.

\bibitem{urbano08}
J.M. Urbano.
\newblock {\em The method of intrinsic scaling}, volume 1930 of {\em Lecture
  Notes in Mathematics}.
\newblock Springer-Verlag, Berlin, 2008.
\newblock A systematic approach to regularity for degenerate and singular PDEs.


\bibitem{vazquez06}
J.L. V{\'a}zquez.
\newblock {\em Smoothing and decay estimates for nonlinear diffusion
  equations}, volume~33 of {\em Oxford Lecture Ser. Math. Appl.}
\newblock Oxford University Press, Oxford, 2006.
\newblock Equations of porous medium type.


\bibitem{wiegner86}
M.~Wiegner.
\newblock On {$C_\alpha$}-regularity of the gradient of solutions of degenerate
  parabolic systems.
\newblock {\em Ann. Mat. Pura Appl. (4)}, 145:385--405, 1986.




\end{thebibliography}
\end{document}